\documentclass[12pt]{article}

\usepackage{graphicx}
\usepackage{amsfonts}
\usepackage{amsmath}
\usepackage{amsthm}
\usepackage{amssymb}
\usepackage{enumitem}
\usepackage{mathrsfs}
\usepackage{multirow}
\usepackage{subcaption}
\usepackage{hyperref}
\hypersetup{
	colorlinks=true,
	linktoc=all,
	linkcolor=blue,
}
\usepackage{geometry}

\graphicspath{{./Figures/}}

\newtheorem{theorem}{Theorem}[section]
\newtheorem{lemma}[theorem]{Lemma}
\newtheorem{corollary}[theorem]{Corollary}

\newtheorem{observation}[theorem]{Observation}
\newtheorem{definition}[theorem]{Definition}

\begin{document}

\title{The structure of graphs with no $K_{3,3}$ immersion}

\author{
   Matt DeVos\thanks{Department of Mathematics, Simon Fraser University, Burnaby, B.C., Canada V5A 1S6, mdevos@sfu.ca.
     Supported in part by an NSERC Discovery Grant (Canada).}
\and
   Mahdieh Malekian\thanks{Department of Mathematics, Simon Fraser University, Burnaby, B.C., Canada V5A 1S6, mahdieh\_malekian@sfu.ca}
}

\date{}

\maketitle

\begin{abstract}
The Kuratowski-Wagner Theorem asserts that a graph is planar if and only if it does not have either $K_{3,3}$ or $K_5$ as a minor.  Using this Wagner obtained a precise description of all graphs with no $K_{3,3}$ minor and all graphs with no $K_5$ minor.  Similar results have been achieved for the class of graphs with no $H$-minor for a number of small graphs $H$.  

In this paper we give a precise structure theorem for graphs which do not contain $K_{3,3}$ as an immersion. This strengthens an earlier theorem of Giannopoulou, Kami\'{n}ski, and Thilikos \cite{MR3283573} that gives a rough description of the class of graphs with no $K_{3,3}$ or $K_5$ immersion.
\end{abstract}


\section{Introduction} 
\label{sec-intro}
In this paper, we will consider finite undirected graphs which may have parallel edges, but no loops (they contribute nothing to the theory for the graphs of interest).  
Graph minors is a well-established part of structural graph theory. One of the first prominent theorems in this area is the Kuratowski-Wagner Theorem, which characterizes planar graphs.
\begin{theorem}[Kuratowski \cite{Kuratowski1930}; Wagner \cite{MR1513158}]
	A graph is planar if and only if it does not contain $K_{3,3}$ or $K_5$ as a minor.
\end{theorem}
Building upon this theorem, Wagner characterized the class of graphs which do not have $K_{3, 3}$ as a minor, and also those which do not contain $K_5$ as a minor \cite{MR1513158}. In addition to these, there are many structural theorems which characterize graphs without an $H$-minor for a fixed graph $H$. Such theorems exist, in particular, for when $H$ is one of  $W_5$, Prism, Octahedron, Cube, and $V_8$, see \cite{MR999699, MR0157370, MR3090713, MR1794690, MR3548299}.

In this paper we consider another graph containment relation, immersion. Let $G$ be a graph and let $xy,yz \in E(G)$ be a pair of distinct edges with a common neighbour. The pair of edges $xy, yz$ is said to \emph{split off at} $y$ if we delete these edges and add a new edge $xz$. If a graph isomorphic to $H$ can be obtained from $G$ by a sequence of splittings and edge and vertex deletions, then we say that $H$ is \emph{immersed} in $G$, or $G$ has an $H$ \emph{immersion}, denoted  $H \prec G$ or $G \succ H$. Equivalently, the graph $G$ has an $H$ immersion if and only if there exists a function $\phi$ with domain $V(H) \cup E(H)$ satisfying the following properties:
\begin{itemize}
	\item $\phi$ maps $V(H)$ injectively to $V(G)$
	\item $\phi$ assigns every $e = uv \in E(H)$ a path $\phi(e) \subseteq G$ with ends $\phi(u)$ and $\phi(v)$
	\item If $e,f \in E(H)$ are distinct, then $\phi(e)$ and $\phi(f)$ are edge-disjoint.
\end{itemize}

We call the vertices in $\phi( V(H) )$ \emph{terminals} of the immersion.  Let us comment that there is another variant of graph immersion, called strong immersion, which has the added restriction that every path $\phi(e)$ must be internally disjoint from the set of terminals.  The type of immersion considered here is also known as weak immersion, however since it is the only relation we consider we will just call it ``immersion''.

In sharp contrast to the setting of graph minors, there are very few results about the structure of graphs with certain particular forbidden immersions.  The most prominent of these is a theorem where both $K_{3,3}$ and $K_5$ are forbidden (in analogy with the Kuratowski-Wagner Theorem).  To state it we require a couple of standard definitions.  If $G = (V,E)$ is a graph and $X \subseteq V$ we let $\delta_G(X)$ denote the set of edges with exactly one endpoint in $X$ and we let $d_G(X) = | \delta_G(X)|$ (we drop these subscripts when the graph is clear from context).  We say that $G$ is $k$-\emph{edge-connected} \emph{(internally} $k$-\emph{edge-connected)} if $d(X) \ge k$ whenever $X, V \setminus X \neq \emptyset$ ($|X|, |V \setminus X| \ge 2$).  We postpone the definition of branch-width to the last section.  

\begin{theorem}[Giannopoulou, Kami\'{n}ski, and Thilikos \cite{MR3283573}]
\label{intro-k33-greek}
If $G$ is a $3$-edge-connected and internally $4$-edge-connected graph with no $K_5$ or $K_{3,3}$ immersion, then $G$ is either cubic and planar, or $G$ has branch-width at most $10$.
\end{theorem}

This theorem is appealing since it splits the possible graphs $G$ satisfying the assumptions into two nice classes: cubic and planar, or branch-width at most 10.  However, this theorem is not a precise characterization since the latter class contains many graphs with $K_{3,3}$ or $K_5$ as an immersion.  Our main result is a precise structure theorem for the class of graphs with no $K_{3,3}$ immersion.  This immediately yields a precise structure for the class of graphs with no $K_{3,3}$ or $K_5$ immersion.  As a corollary we show that the bound on branch-width in the above theorem may be reduced to~3.

In the statement of Theorem \ref{intro-k33-greek} we have assumed that the graph $G$ is 3-edge-connected and internally 4-edge-connected.  However, this result can be meaningfully applied without any connectivity assumption on $G$.  Before explaining, let us insert a bit of terminology:  If $H$ is a graph and $X \subset V(H)$ then we let $H.X$ denote the graph obtained from $H$ by identifying all vertices in $X$ to a single new vertex (any loops created in this process are deleted).  Now consider an arbitrary graph $G$ for which we are interested in determining the presence of a $K_{3,3}$ immersion.  
Let $G'$ be obtained from $G$ by deleting every cut-edge and note that $K_{3,3} \prec G$ if and only if $K_{3,3} \prec G'$.  Next, consider a component $H$ of $G'$ with a set $X \subset V(H)$ satisfying $d(X) = 2$.  Form a new graph $H'$ ($H''$) from $H.X$ ($H.(V(H) \setminus X)$) by suppressing the newly created degree two vertex.  Then modify the graph $G'$ by removing the component $H$ and then adding the components $H'$ and $H''$.  If we repeat this modification until it can no longer be applied, the resulting graph $G''$ will again satisfy $K_{3,3} \prec G$ if and only if $K_{3,3} \prec G''$, but now every component of $G''$ is 3-edge-connected.  Finally, suppose that $H$ is a component of $G''$ with a set $X \subset V(H)$ satisfying $d(X) = 3$ and $|X|, |V(H) \setminus X| \ge 2$.  In this case we delete the component $H$ and add the components $H.X$ and $H.(V(H) \setminus X)$.  If this operation is repeated until no longer possible, the resulting graph $G'''$ will still satisfy $K_{3,3} \prec G$ if and only if $K_{3,3} \prec G'''$, however every component of $G'''$ will be 3-edge-connected and internally 4-edge-connected. 
A similar argument holds for $K_5$ too. 
 Accordingly, we will focus our attention on 3-edge-connected and internally 4-edge-connected graphs in the rest of the paper.

Before we can describe our structure theorem we require a few additional ingredients.  Namely we will need to introduce a certain reduction operation and a very tightly structured class of graphs which have small nested edge-cuts. We call a graph $H$ a \emph{doubled cycle} \emph{(doubled path)} if it can be obtained from a cycle (path) by adding a second copy of each edge.  

\begin{definition}
\normalfont
Let $G$ be a graph and let $X \subset V(G)$ satisfy $ |X| =k \ge 2$. We say $G[X]$ is a {\it chain of sausages of order} $k$ in $G$ if 
$G. (V(G)\setminus X )$ is a doubled cycle (of length $k+1$).  
If $G[X]$ is a chain of sausages of order at least 3 and $x,x' \in X$ are adjacent, then the operation of identifying $x$ and $x'$ to a new vertex is called a \emph{sausage shortening}.  If $G'$ is obtained from $G$ by repeatedly performing sausage shortenings until this operation is no longer possible, we call $G'$ a \emph{sausage reduction} of $G$.  We say that $G$ is \emph{sausage reduced} if it has no sausage of order at least 3 (so no sausage reduction is possible).  
\end{definition} 

\begin{figure}[htbp]
	\centering
	\includegraphics{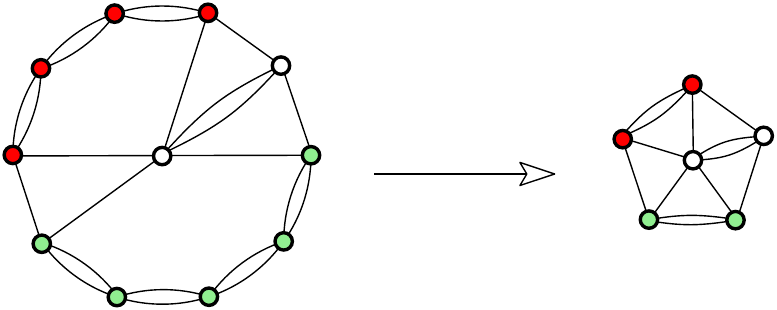}
	\caption{Sausage reduction}
	\label{fig:s-red}
\end{figure}

\begin{definition}
\normalfont
If $G$ is a graph and $X,Y \subseteq V(G)$ are disjoint, a \emph{segmentation} of $G$ \emph{relative to} $(X,Y)$ is a family of nested sets $X_1 \subset X_2 \subset X_3 \ldots \subset X_k$ satisfying:
\begin{itemize}
\item $X_1 = X$ and $V(G) \setminus X_k = Y$
\item $|X_{i+1} \setminus X_i| = 1$ for $1 \le i \le k-1$
\end{itemize}
We say that the segmentation has \emph{width} $k$ if $d(X_i) = k$ holds for every $1 \le i \le k$ and we call it an $(a,b)$-\emph{segmentation} if $|X| \le a$ and $|Y| \le b$.  
\end{definition}

	\begin{figure}[htbp]
		\centering
		\includegraphics{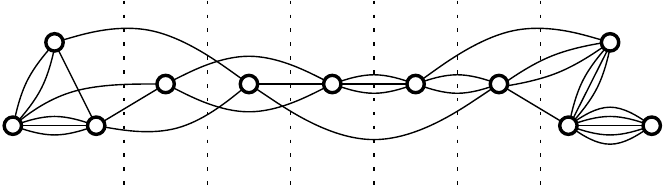}
		\caption{A graph with a $(3, 3)$-segmentation of width four}
		\label{fig:segmentation-ex}
	\end{figure}

Our full structure theorem features a number of sporadic small graphs that have no $K_{3,3}$ immersion and is stated in Section  \ref{sec-k33}.  For sausage reduced graphs on at least nine vertices our result may be stated as follows.

\begin{theorem}
	\label{intro-k33-thm}
	Let $G$  be a $3$-edge-connected and internally $4$-edge-connected graph which is sausage reduced.  If $|V(G)|\ge 9$, then $G$ does not immerse $K_{3, 3}$ if and only if
	\begin{itemize}
		\item $G$ is planar and cubic, or
		\item $G$ has a $(3, 3)$-segmentation of width four.
	\end{itemize}
\end{theorem}

Note that planar graphs and graphs with a $(3,3)$-segmentation of width four cannot contain a $K_5$ immersion.  Accordingly, the outcome of this theorem remains unchanged if we forbid both $K_{3,3}$ and $K_5$ as immersions.  

Our proof of the structure theorem for graphs with no $K_{3,3}$ immersion requires two theorems concerning rooted immersions of smaller graphs.  One of these is established in an earlier paper 
\cite{devosforbidden1}, the other appears in the second section of this paper.  In addition we have called upon a computer to establish the result on all graphs with at most nine vertices.  This has the great advantage of letting the computer discover all of the sporadic exceptional graphs and permitting us a clean proof for larger graphs which need not even encounter these small obstructions.  

The rest of this paper is organized as follows: In Section \ref{sec-eye}, we will introduce and see a result about the immersion of a certain rooted graph, called Eyeglasses, and in Section \ref{sec-k33} we will state and prove our main result on the structure of graphs excluding $K_{3,3}$ as immersion.

\section{Forbidding Eyeglasses}
\label{sec-eye}

Precise structural theorems for small rooted graphs can be extremely useful tools in finding immersions of somewhat bigger graphs. They enable us to break a bigger graph into smaller rooted pieces and then tie together the immersions of these smaller pieces to obtain an immersion of the desired bigger graph.  The purpose of this section is to establish a precise structure theorem for graphs that do not contain a certain rooted graph called \emph{Eyeglasses} as an immersion.  We begin with some terminology.  

A \emph{rooted graph} consists of a connected graph $G$ together with an ordered tuple $(x_1, \ldots, x_k)$ of distinct vertices called \emph{roots}.  If $H$ together with $(y_1, \ldots, y_k)$ is another rooted graph, we say $G$ contains $H$ as a {\it rooted immersion} if there is a sequence of splits and edge-deletions which transforms $G$ into a graph isomorphic to $H$, where this isomorphism sends $x_i$ to $y_i$, for $i=1,\ldots, k$.  We will write $(G; x_1, \ldots ,x_k) \succ_r (H; y_1, \ldots , y_k)$ to indicate that $G$ has a rooted immersion of $H$. For clarity, in our figures the roots are always solid while other vertices are open.  We define Eyeglasses to be a rooted graph obtained from a path of length three in which the edges incident with the two ends of the path are doubled, and the two ends of the paths are the roots, see Figure \ref{fig:eyeglasses}.
\begin{figure}[htbp]
	\centering
	\includegraphics{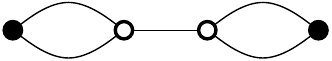}\\
	\caption{Eyeglasses}
	\label{fig:eyeglasses}
\end{figure}

Next we introduce three classes of graphs which do not immerse Eyeglasses.  For a vertex $v$ in a graph $G$, we let $d(v)$ denote the degree of $v$ and let $N(v)$ denote the set of vertices adjacent to $v$.  
\begin{definition} 
	\normalfont
	Let $G = (V, E)$ together with $(x_0,x_1)$ be a rooted graph and assume $d(x_0) = d(x_1) = 2$.  We say that $G$ is 
\end{definition}
\begin{description}
	\item [Type 1] if $d(v) = 3$ for every $v \in V \setminus \{x_0,x_1\}$, each root has two neighbours, say $N(x_i) = \{s_i, t_i\}$, where $s_0,s_1,t_0,t_1$ are distinct, and there is a planar embedding of $G\setminus  \{x_0, x_1\}$ in which $s_0, s_1, t_0, t_1$ appear on the boundary of the outer face, in this cyclic order (see Figure \ref{fig:eyeglassesobs1}).
	\item [Type 2] if there exist two distinct vertices $u, v\in V \setminus \{x_0, x_1\}$ such that $\{x_0, x_1\}= N(u)= N(v)$, and $G \setminus \{x_0, x_1\}$ is a $u-v$-doubled path (see Figure \ref{fig:eyeglassesobs2}).
	\item [Type 3] if $G$ is isomorphic to the graph obtained from $K_4$ by deleting an edge.  (see Figure \ref{fig:eyeglassesobs3}).
	\end{description}
\begin{figure}[htbp]
	\centering
	\begin{subfigure}[b]{0.35\textwidth}
		\centering
		\includegraphics{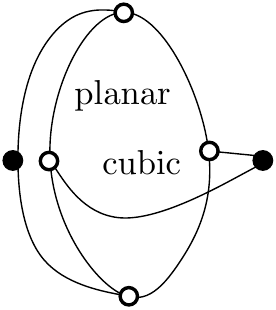}
		\caption{Type 1}
		\label{fig:eyeglassesobs1}
	\end{subfigure}
	\hfill
	\begin{subfigure}[b]{0.35\textwidth}
		\centering
		\includegraphics{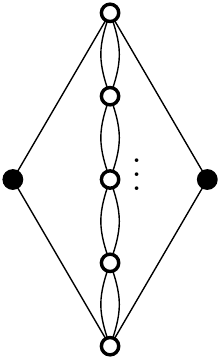}
		\caption{Type 2}
		\label{fig:eyeglassesobs2}
	\end{subfigure}
	\hfill
	\begin{subfigure}[b]{0.25\textwidth}
		\centering
		\includegraphics{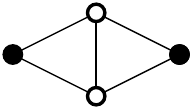}
		\caption{Type 3}
		\label{fig:eyeglassesobs3}
	\end{subfigure}
	\caption{Obstructions to an immersion of Eyeglasses}
	\label{fig:types}
\end{figure}

Type 1 is an especially interesting obstruction since a topological condition appears here.  In a graph $G$ of this type, every vertex has degree at most three, so any immersion of Eyeglasses would require $G$ to have two vertex disjoint cycles $C_0, C_1$ with $x_i \in V(C_i)$ for $i=0,1$.  However the above topological condition implies that no such cycles can exist.  To handle this case we will need to call on a characterization of graphs without a 2-linkage.  This theorem has appeared by several authors in varying forms (see also \cite{MR560773}), we will use that given by Thomassen.

\begin{theorem}[Thomassen \cite{MR595938}]
	\label{tpp}
	Let $G$ be a graph and let $T = \{s_0, t_0, s_1, t_1\} \subseteq V(G)$.  Assume that for every $X \subseteq ( V(G) \setminus T)$ with $|X| = 3$ the graph $G \setminus X$ has at most one component disjoint from $T$, and if such a component exists it is an isolated vertex.  Then one of the following holds:
\begin{itemize}
\item There exist vertex disjoint paths $P_0, P_1$ in $G$ so that $P_i$ has ends $s_i, t_i$ for $i=0,1$.
\item There is a planar embedding of $G$ in which $s_0, s_1, t_0, t_1$ appear on the boundary of the infinite face in this cyclic order.
	\end{itemize}
\end{theorem}

%

Our main result from this section is the following.

\begin{theorem}
	\label{eye-thm}
Let $G=(V,E)$ together with $(x_0,x_1)$ be a rooted graph and assume $|V| \ge 4$ and $d(x_0) = d(x_1) = 2$.  If
\begin{enumerate}[label=(\roman*)]
\item \label{eyes2c} $d(X) \ge 2$ whenever $X \subseteq V$ satisfies $x_0 \in X$ and $x_1 \not\in X$
\item \label{eye3ec} $d(Y) \ge 3$ whenever $Y \subseteq V$ satisfies $x_0,x_1 \not\in Y$ and $Y \neq \emptyset$
\item \label{eyei4c} $d(Y) \ge 4$ whenever $Y \subseteq V$ satisfies $x_0,x_1 \not\in Y$ and $|Y| \ge 2$
\end{enumerate}
%
%
Then $G$ has a rooted immersion of Eyeglasses if and only if $G$ is not of one of the types  $1, 2$, or $3$.
\end{theorem}
%
%

%

\begin{proof} 
It is straightforward to verify that the graphs of type 1, 2 and 3 do not have an immersion of Eyeglasses. For the reverse direction, suppose (for a contradiction) that $G=(V, E)$ is a counterexample to the theorem.  

First suppose that every $v \in V \setminus \{x_0, x_1\}$ satisfies $d(v) = 3$.  If there exists $y \in N(x_0) \cap N(x_1)$, then $d( \{x_0, x_1, y \}) \le 3$ and $G$ must be type 3 by \ref{eyei4c}.  If there exist two vertex disjoint paths $P_0, P_1$ so that $P_i$ has ends $N(x_i)$ for $i=0,1$ (here we permit the possibility that $|N(x_i)| = 1$ and $P_i$ is trivial), then the connectivity of $G$ immediately implies that $G$ has a rooted immersion of Eyeglasses.  In the remaining case, let $N(x_i) = \{s_i, t_i\}$ for $i=0,1$ and let $T = \{s_0, s_1, t_0, t_1\}$.  It follows from \ref{eye3ec} and \ref{eyei4c} (and the fact that $G$ is subcubic) that $G - \{x_0, x_1\}$ satisfies the hypothesis of Theorem \ref{tpp}, and therefore $G$ has type 1.  

By the above we may assume that $G$ contains a vertex $v$ with $d(v) = 4$.  Using \ref{eyes2c} and \ref{eyei4c} we may choose four edge-disjoint paths $P_1, P_2, P_3, P_4$ starting at $v$, so that $P_1, P_2$ end at $x_0$ and $P_3, P_4$ end at $x_1$. (To see that this is possible, add a new vertex $s$ and two edges between $s$ and $x_i$ for $i=0,1$, and then apply Menger's Theorem to find four edge-disjoint paths from $v$ to $s$).  We let ${\mathcal B}$ denote the set of all components of $G - \bigcup_{i=1}^4 E(P_i)$ with at least two vertices.  Every $B \in {\mathcal B}$ is called a \emph{bridge} and a vertex in $V(B) \cap \left(\bigcup_{i=1}^4 V(P_i) \right)$ is called an \emph{attachment}.  Let $V_i$ denote the set of internal vertices of $P_i$ for $1 \le i \le 4$ and note that the degree assumptions on $x_0, x_1$ imply that $x_0, x_1\notin \bigcup_i V_i$ and that $x_0$ and $x_1$ do not appear in any bridge.  Next we establish a sequence of properties of $G$.

\begin{enumerate}[label=(\arabic*), labelindent=0em ,labelwidth=0cm, parsep=6pt, leftmargin =7mm]

\item
\label{vxidentificationV1V2}
$V_1\cap V_2= \emptyset = V_3 \cap V_4$.
			
If the vertex $x$ were to be contained in either $V_1 \cap V_2$ or $V_3 \cap V_4$, then $P_1\cup\ldots \cup P_4 $ would have an immersion of Eyeglasses using the terminals $x_0, x_1, v, x$.   

\item There does not exist a bridge with a unique attachment.
\label{nosingleb}

Suppose (for a contradiction) that $B$ is a bridge with a unique attachment vertex $x$.  Choose $y \in V(B) \setminus \{x\}$ and note that by \ref{eye3ec} we may choose three edge-disjoint paths $Q_1, Q_2, Q_3$ from $x$ to $y$.  Either $x \in \cup_{i=1}^4 V_i$ or $x = v$. Note that by \ref{eyei4c}, in the former case $x$ must appear on at least two $V_i$'s, so we see that 
in either case $\big( \bigcup_{i=1}^4 P_i\big) \cup \big( \bigcup_{j=1}^3 Q_j \big)$ has a rooted immersion of Eyeglasses with terminals $x_0, x_1, x, y$, giving us a contradiction.  

\item The vertex $v$ is not an attachment of any bridge.
\label{novatb}

Suppose (for a contradiction) that $B \in {\mathcal B}$ has $v$ as an attachment, and note that by \ref{nosingleb} we may choose another attachment $x$ of $B$.  Now $B$ contains a path $Q$ from $x$ to $v$ and $Q \cup \big( \bigcup_{i=1}^4 P_i\big)$ has a rooted immersion of Eyeglasses with terminals $x_0, x_1, x, v$, giving us a contradiction.

\item There does not exist a bridge with an attachment in $V_i$ and an attachment in $V_j$ where $i \neq j$.
\label{ntrivialViVj}

Suppose (for a contradiction) that $B$ has attachment vertices in both $V_i$ and $V_j$ where $i \neq j$ and note that by \ref{nosingleb} we may choose distinct attachment vertices $x \in V_i$ and $y \in V_j$.  Now $B$ contains a path $Q$ from $x$ to $y$ and then $Q \cup \big( \bigcup_{i=1}^4 P_i\big)$ has a rooted immersion of Eyeglasses, giving us a contradiction.

\item 
\label{middleinV1}
Suppose $B \in {\mathcal B}$ has distinct attachments $w,w' \in V_j$.  If $P_j'$ is the subpath of $P_j$ from $w$ to $w'$, then $V(P'_j) \cap \left (\bigcup _{i \neq j}  V_i \right ) = \emptyset$.
			
Suppose (for a contradiction) that the above is violated and choose $x \in V(P_j')$ with $x \in V_i$ for some $i \neq j$.  
Let $ Q $ be a path in $B$ from $w$ to $w'$ and note that $x \not\in V(Q)$ (otherwise $B$ would contradict \ref{ntrivialViVj}).  Now using \ref{vxidentificationV1V2} we find that $P_1\cup\ldots \cup P_4\cup Q $ has a rooted immersion of Eyeglasses.

\item 
\label{nobridges}
${\mathcal B} = \emptyset$.

Suppose for a contradiction that ${\mathcal B} \neq \emptyset$ and suppose that $V_j$ contains attachment vertices of some bridge $B$.  Using \ref{middleinV1} we may define $P_j'$ to be the maximal subpath of $P_j$ with the property that $P_j'$ contains all attachments of $B$ and $P_j'$ does not contain any vertex in $\bigcup_{i \neq j} V(P_i)$ (so in particular $x_0, x_1, v \not\in V(P_j')$).  It now follows from \ref{middleinV1} that every bridge with an attachment in $P_j'$ has all attachments in $P_j'$.  However, now the subgraph $H$ consisting of $P_j'$ together with all bridges with attachments on this path has $d_G( V(H) ) = 2$, and this contradicts the edge-connectivity of $G$.
		
\item
\label{vxidentificationV1V3}
For every $1 \le i \le 4$ there is at most one $1 \le j \le 4$ with $i \neq j$ so that $V_i \cap V_j \neq \emptyset$.
		
Suppose (for a contradiction) that $x \in V_1\cap V_3$ and $y \in V_1\cap V_4$ (the other cases are similar). Note by \ref{vxidentificationV1V2} we have $ V_3\cap V_4=\emptyset $, and thus $ x \neq y $. Now $P_1\cup \ldots \cup P_4$ has a rooted immersion of Eyeglasses using the terminals $x_0, x_1, v$ and either $x$ or $y$, a contradiction.

\item
\label{vxidentificationV1V3order}
If $V_i \cap V_j$ is nonempty, then the order of appearing the vertices in $V_i\cap V_j$ on $P_i, P_j$ is the same.
		
We suppose $i=1$ and $j=3$, the other cases are similar.  If there exist $x, y \in V_1\cap V_3$ so that $x$ appears before $ y $ on $ P_1 $, but after $ y $ on $ P_3 $, then $P_1\cup\ldots \cup P_4 $ would contain an Eyeglasses immersion on $x_0, x_1, x, y$, a contradiction. 

\end{enumerate}

It follows from \ref{vxidentificationV1V2}, \ref{nobridges}, \ref{vxidentificationV1V3}, and \ref{vxidentificationV1V3order} that $G$ has type 2, and this final contradiction completes the proof.		
	\end{proof}
\section{Forbidding $K_{3,3}$}
\label{sec-k33}
In this section we state and prove a precise structure theorem for graphs which do not contain $K_{3,3}$.

\subsection{Statement of the main theorem}
\label{sub-k33-pre}
We start by introducing five families of graphs which do not immerse $K_{3,3}$.

\begin{description}
	\item [Type 0.] $G$ is type 0 if it is planar and cubic.
	\item [Type 1.] $G$ is type 1 if it has a $(3,3)$-segmentation of width four.
	
	\item [Type 2.] $G$ is type 2 if  there exist disjoint sets $W, W' \subseteq V(G)$ with $1 \le |W|, |W'| \le 2$ such that the graph $G^*$ obtained by identifying $W$ ($W'$) to a single vertex $w$ ($w'$) has a doubled cycle $C$ containing $w, w'$ satisfying one of the following:
	\begin{description}
		\item[(2A)] $w$ and $w'$ are not adjacent in $C$ and $G^* = C + ww'$ (see Fig. \ref{fig:k33-2A})
		\item[(2B)] $w$ and $w'$ have a common neighbour $v$ in $C$ and $G^* = C + wv + vw'$ (see Fig. \ref{fig:k33-2B})
		\item[(2C)] $w$ and $w'$ are adjacent in $C$ and $G^* = C + ww'$ (see Fig. \ref{fig:k33-2C})
	\end{description}
	\begin{figure}[htbp]
		\centering
		\begin{subfigure}[b]{0.3\textwidth}
			\centering
			\includegraphics{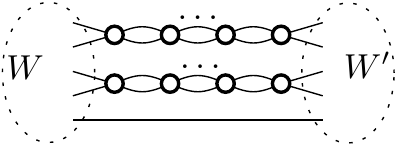}
			\caption{Type $2A$}
			\label{fig:k33-2A}
		\end{subfigure}
		\hfill
		\begin{subfigure}[b]{0.3\textwidth}
			\centering
			\includegraphics{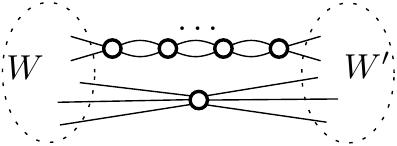}
			\caption{Type $2B$}
			\label{fig:k33-2B}
		\end{subfigure}
		\hfill
		\begin{subfigure}[b]{0.3\textwidth}
			\centering
			\includegraphics{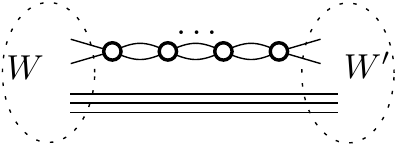}
			\caption{Type $2C$}
			\label{fig:k33-2C}
		\end{subfigure}
		\hfill
		\caption{Type 2 graphs}
		\label{fig:type2}
	\end{figure}
	\item [Type 3.] $G$ is type 3 if after sausage reduction it is isomorphic to one of the 20 graphs in Figure \ref{k33-exceptions-1}. That is $G$ is type 3 if it can be obtained from a graph in Figure \ref{k33-exceptions-1} by replacing  any pair of same-colored (not white) vertices with a chain of sausages of arbitrary order ($\ge 2$).
	
	\begin{figure}[htbp]
		\centering
		\includegraphics[height=5.2cm]{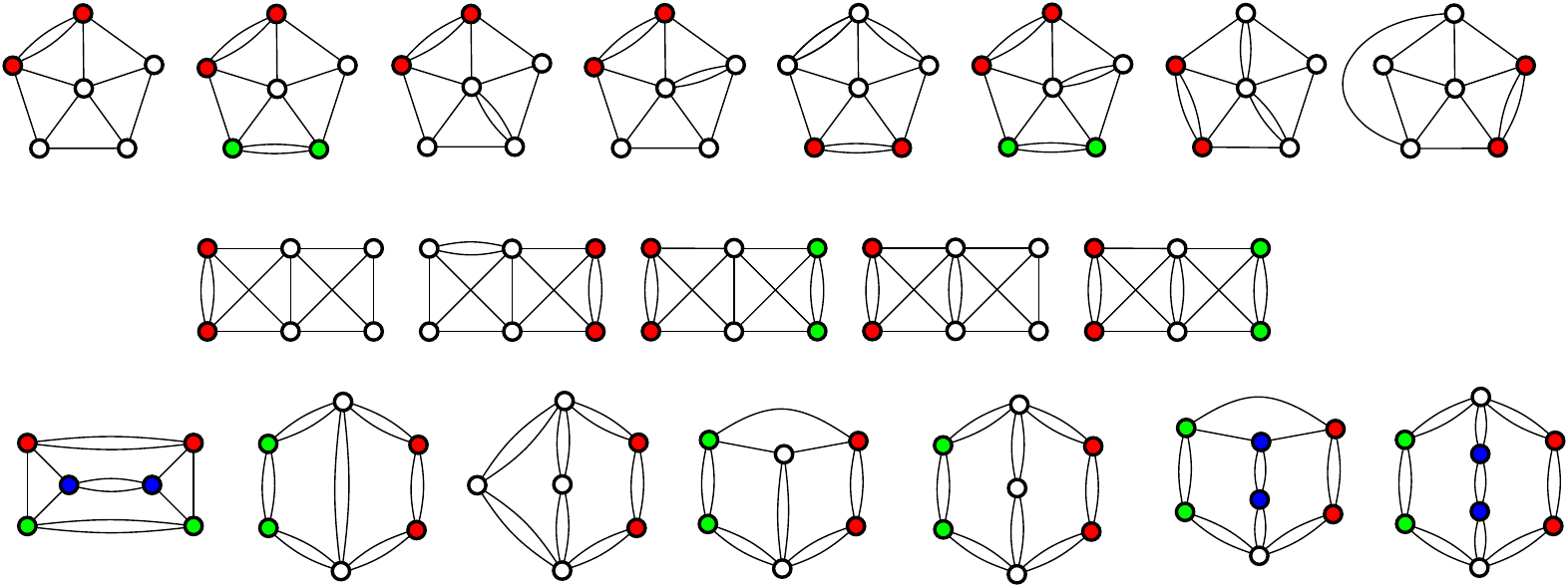}
		\caption{Graphs of type 3 after sausage reduction}
		\label{k33-exceptions-1}
	\end{figure}
	\item [Type 4.] $G$ is type 4 if it is isomorphic to one of the 14 graphs in Figure \ref{k33-exceptions-2}.
	\begin{figure}[htbp]
		\centering
		\includegraphics[height=3.1cm]{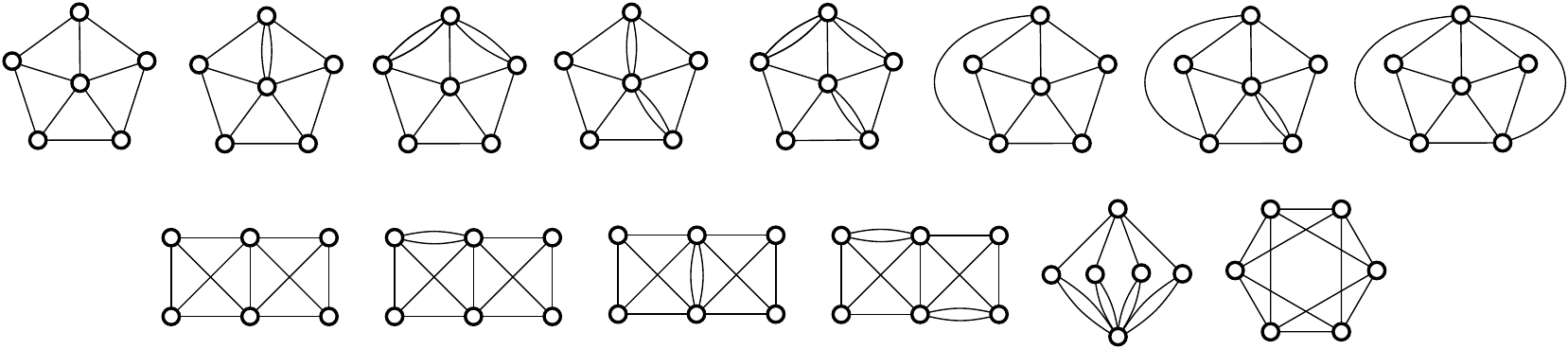}
		\caption{Type 4 graphs}
		\label{k33-exceptions-2}
	\end{figure}
\end{description}

With these definitions of types, we can now state our characterization of graphs with no $K_{3,3}$ immersion as follows:
\begin{theorem}
	\label {k33-thm}
	Let $G$ be a $3$-edge-connected, and internally $4$-edge-connected graph  with $|V(G)|\ge 6$. Then $G$ does not immerse $K_{3,3}$ iff $G$ is not one of the types $0, 1, 2, 3$, or $4$.
\end{theorem}

Suppose $G$ has one of the types $0, 1, 2, 3$, or $4$. Observe that if $G$ is sausage reduced and $|V(G)| \ge 9$, then $G$ is either type 0, or type 1. This shows that Theorem \ref{intro-k33-thm} is in fact a corollary of the above theorem. 

\subsection{Proof of the `if' direction}

We record a few simple properties before beginning the easy direction of the proof.

\begin{observation}
\label{containk33obs}
Suppose that $G$ is a graph that contains an immersion of $K_{3,3}$ with terminals~$T$.  Then we have
\begin{enumerate}
\item \label{sausageok} the graph obtained from $G$ by sausage reduction immerses $K_{3,3}$, and
\item \label{t1nok33} if $G$ has a segmentation $X_0 \subset X_1 \subset \ldots \subset X_k$ of width $4$ with $|X_0| \le 3$, then $|T \cap X_k| \le 2$, and
\item \label{k33imparity} if $v \in T$ has $d(v)$ even, then there is an edge $e$ incident with $v$ so that $G-e$ immerses $K_{3,3}$.
\end{enumerate}
\end{observation}

\begin{proof}
The first two parts follow from the fact that every set $X \subseteq V(K_{3,3})$ with $|X| = 3$ satisfies $d(X) \ge 5$.  The last is immediate from our definitions.
\end{proof}

\begin{proof}[Proof of the `if' direction of Theorem \ref{k33-thm}]
We will show that graphs of type $0, 1, 2, 3$ and 4 do not immerse $K_{3, 3}$.  First suppose $G$ is type 0. Since $G$ is cubic, it has a  $K_{3,3}$ immersion iff it has a $K_{3,3}$-subdivision.  However, this latter possibility is forbidden by planarity.  Graphs of type 1 cannot immerse $K_{3,3}$ by part \ref{t1nok33} of the previous observation.  For graphs of type 3, part \ref{sausageok} of the previous observation reduces our task to verifying that the graphs in Figure \ref{k33-exceptions-1} do not have $K_{3,3}$ immersions.  Similarly for graphs of type 4 we must verify that the graphs in Figure \ref{k33-exceptions-2} do not immerse $K_{3,3}$.  This verification can be done by hand, but we have used a computer.\footnote{The code is available on the arXiv.}

To finish the proof of the `if' direction we must prove that type 2 graphs do not immerse $K_{3,3}$.  Suppose (for a contradiction) that $G$ is type 2 relative to $W,W'$ and $G \succ K_{3,3}$, and let $T$ be the terminals of a $K_{3,3}$ immersion in $G$.  By part \ref{sausageok} of the previous observation, we may assume that $G$ is sausage reduced. Next suppose that $H$ is a chain of sausages of order two in $G$, say $V(H) = \{u,v\}$ with $u,v \not\in W \cup W'$ and that $u,v \in T$.  Part \ref{k33imparity} of the previous observation implies that $G$ will still immerse $K_{3,3}$ even after deleting either one copy of the edge $uv$ or both an edge incident with $u$ and one incident with $v$.  The latter possibility is impossible by the 3-edge-connectivity of $K_{3,3}$.  In the former case $G- uv$ has $K_{3,3}$ as an immersion, but this is also not possible since $G- uv$ is type 1.  Therefore at most one terminal of $K_{3,3}$ lies on any chain of sausages disjoint from $W \cup W'$.

So, if we let $G'$ be the graph obtained from $G$ by splitting off each chain of sausages (if existent) disjoint from $W \cup W'$ down to only one vertex, then $G' \succ K_{3, 3}$. This immediately gives a contradiction in the cases where either $G$ has type 2C, or $|W| =1$, or $|W'| =1$, as then $|V(G')| \le 5$.  In the remaining case $G$ has type 2A or 2B, and $|W|= |W'| =2$. Since $|V(G')| = 6$, every vertex of $G'$ must be a terminal in any immersion of $K_{3,3}$.  However, then part \ref{k33imparity} of the previous observation allows us to remove an edge incident with each vertex in $V(G') \setminus (W \cup W')$ while preserving a $K_{3,3}$ immersion.  The graph resulting from this operation would have an internal 3-edge-cut, and this contradiction shows that $G$ cannot immerse $K_{3,3}$.  
\end{proof}

\subsection{Four edge cuts}
\label{sub-k33-tool}

\label{k33-pf-start}
Our goal for this section is to show that a minimal counterexample to Theorem \ref{k33-thm} cannot have a 4-edge-cut with at least three vertices on each side.  The proof of this will call upon another structure theorem by the authors.  The graph $W_4$ is a simple graph obtained from a cycle of length 4 by adding a new vertex $u$ and an edge between $u$ and each existing vertex.  We turn $W_4$ into a rooted graph by declaring $u$ to be the root vertex.  

\begin{theorem}[DeVos, Malekian \cite{devosforbidden1}]
	\label{r-w4-thm-k33}
	Let $G$ be a graph with $|V(G)| \ge 5$ and with a root vertex $x$, where $d(x) \in \{4, 5\}$. If $G$ is $3$-edge-connected and internally $4$-edge-connected, then $G$ contains a rooted immersion of $W_4$ if and only if $G$ does not have one of the following types:
	\begin{description}
		\item [Type I. ] $G$ is type I if it has a $(2, 3)$-segmentation of width $4$ relative to $(X,Y)$ where $x \in X$.  
		\item [Type II.] $G$ is type II if there exists $W \subset V(G) \setminus \{x\}$ with $1 \le |W| \le 2$ such that the graph $G^*$ obtained by identifying $W$ to a single vertex $w$ has a doubled cycle $C$ containing $x, w$ which satisfies one of the following:
		\begin{description}
			\item [(II A)] $x$ and $w$ are not adjacent in $C$ and $G^* = C + xw$
			\item [(II B)] $x$ and $w$ have a common neighbour $v$ in $C$ and $G^* = C + xv + vw$
			\item [(II C)] $x$ and $w$ are adjacent in $C$ and $G^* = C + xw$. Moreover we have $|W| =2$.
		\end{description}
	\end{description}
\end{theorem}

Observe that conclusion of the above theorem can be strengthened somewhat under the added assumption $d(x) = 4$.  In this case $G$ cannot be type II.  Furthermore, if $G$ is type I, then (by possibly prepending the set $\{x\}$ to our segmentation) we find that $G$ has a $(1,3)$-segmentation relative to $( \{x\}, Y)$ for some $Y$.

\begin{corollary} 
\label{w4cor}
Let $G$ be a $3$-edge-connected and internally $4$-edge-connected graph with a root vertex $x$ and assume that $G \nsucc_r W_4$.   
\begin{enumerate} 
\item If $d(x) = 4$ then $G$ has a $(1,3)$-segmentation of width $4$ relative to $( \{x\}, Y)$ for some $Y$.
\item Suppose $d(x) = 5$ and $G$ does not have a vertex $v\notin N(x)$ satisfying $d(v) = 4$, and $|N(v)|<4$. Then $G$ is either type I or $|V(G)| \le 5$.  
\end{enumerate}
\end{corollary}

\begin{proof} For the first part, since type II graphs have root vertices of degree five, the previous theorem implies that $G$ has a $(2, 3)$-segmentation of width four, with $x$ in its head. Moreover, $d(x) =4$ implies that $G$ in fact has a $(1, 3)$-segmentation of width four; also observe that a $(1,3)$-segmentation trivially exists when $|V(G)| \le 4$.  The second part follows from the theorem and the observation that any graph of type II with such a property has at most $5$ vertices.  
\end{proof}

Next we take advantage of this $W_4$ theorem to establish a key lemma.  If $H$ is a graph and $X \subset V(H)$ then we will interpret $H.X$ as a rooted graph (when convenient) where it is assumed that the vertex created by identifying $X$ is the root.

\begin{lemma}
	\label{w4-eye-k33}
	Let $G = (V, E)$ be a $3$-edge-connected and internally $4$-edge-connected graph with $G \nsucc K_{3, 3}$. Let $X \subset V$ satisfy $|X|, |V \setminus X| \ge 3$ and $d(X) = 4$.  If the rooted graph $G.X$ has a rooted immersion of $W_4$, then one of the following occurs:
	\begin{enumerate}
		\item
		\label{w4-eye-k33-1}
		$G[X]$ is a chain of sausages, or
		\item 
		\label{w4-eye-k33-2}
		for every vertex $v \in X$ we have $d_G(v) = 3$, and $G. (V\setminus X)$ has a rooted immersion of $W_4$.
	\end{enumerate}
\end{lemma}

\begin{proof}
Denote the root vertex of $G' = G. X$ by $a$.  Let the terminals of a rooted immersion of $W_4$ in $G'$ be $\{a, v_1, v_2, v_3, v_4 \}$, where there is an immersion of $C_4$ on $v_1 v_2 v_3 v_4 v_1$ in this cyclic order (see Figure \ref{fig:w4-eye-k33-Gp}). Let $P_{a v_i}$ be the path in $G'$ corresponding to the $av_i$ edge of $W_4$, and let $e_i= E (P_{av_i}) \cap \delta_{G'} (a)$, for $1 \le  i\le  4$. Now we define $G''$ to be the rooted graph obtained from $G$ by subdividing $e_1, e_3$ ($e_2, e_4$) with a new vertex, and then identifying the degree two vertices to a new vertex $b$ ($c$), see Figure \ref{fig:w4-eye-k33-Gz}. Let $G^* = G''[X \cup \{b, c\}]$ be a rooted graph with roots $(b, c)$.
	
\begin{figure}[htbp]
	\centering
	\begin{subfigure}[b]{0.49\textwidth}
		\centering
		\includegraphics[height=2.5cm]{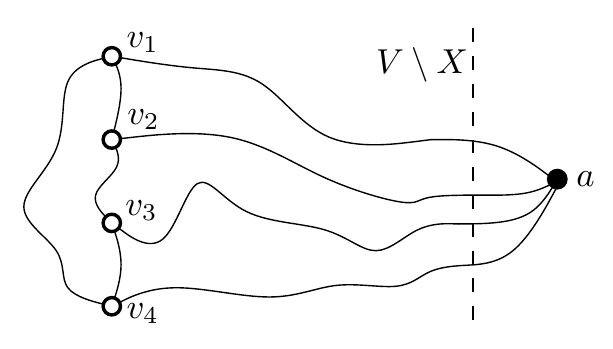}
		\caption{$G'\succ_r W_4$}
		\label{fig:w4-eye-k33-Gp}
	\end{subfigure}
	\hfill
	\begin{subfigure}[b]{0.49\textwidth}
		\centering
		\includegraphics[height=2.5cm]{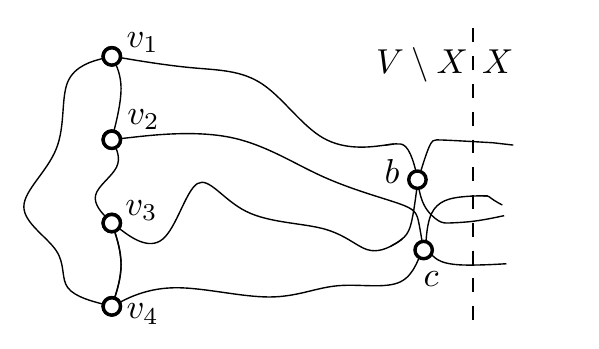}
		\caption{Graph $G''$}
		\label{fig:w4-eye-k33-Gz}
	\end{subfigure}
	\hfill
	\begin{subfigure}[b]{0.33\textwidth}
		\centering
		\includegraphics[height=2.5cm]{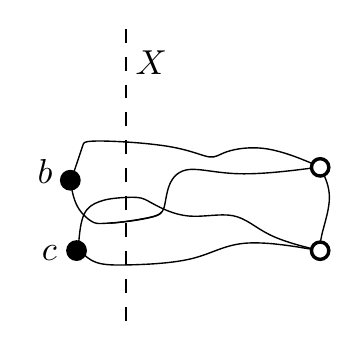}
		\caption{Immersion of Eyeglasses in $G^*$}
		\label{fig:w4-eye-k33-Gs}
	\end{subfigure}
	\hfill
	\begin{subfigure}[b]{0.66\textwidth}
	\centering
	\includegraphics[height=2.5cm]{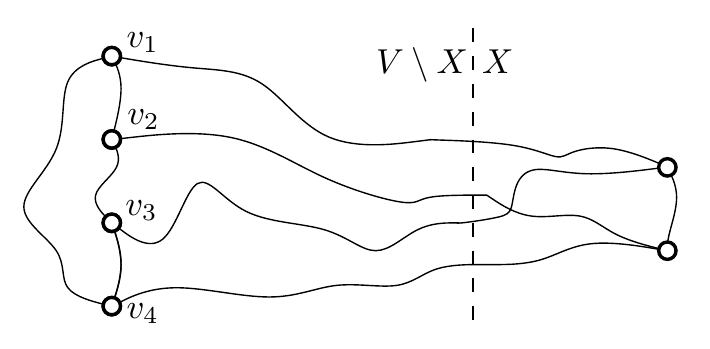}
	\caption{$G' \succ_r W_4$ and $G^* \succ_r$ Eyeglasses implies $G\succ K_{3,3}$}
	\label{fig:w4-eye-k33-G}
    \end{subfigure}
	\caption{}
\end{figure}	
	
	Observe that if there is a rooted immersion of Eyeglasses in $G^*$ (as in Fig. \ref{fig:w4-eye-k33-Gs}), then we have the contradiction $G\succ K_{3,3}$ (see Fig. \ref{fig:w4-eye-k33-G}). It follows from the internal 4-edge-connectivity of $G$ that $G[X]$ does not have a cut-edge. Therefore, by Theorem \ref{eye-thm} (and the observation $|V(G^*)| \ge 5$) we conclude that $G^*$ is either a type 1 or type 2 obstruction to immersion of Eyeglasses.  In the latter case $G[X]$ is a chain of sausages, as desired.  In the former case every $v \in X$ satisfies $d_G (v) = 3$ and $G[X]$ has a planar embedding for which all vertices incident with an edge of $\delta(X)$ are on the unbounded face. It follows from the 2-edge-connectivity of $G[X]$ and the assumption that $G[X]$ has maximum degree at most 3 that $G[X]$ is 2-connected.  Therefore, the unbounded face in our planar embedding of $G[X]$ is bounded by a cycle and we conclude that $G. (V\setminus X)$ has a rooted immersion of $W_4$ as desired.
\end{proof}

%

The next lemma is a helpful tool in proving Lemma \ref{no-deep-4ecut}.

\begin{lemma}
	\label {deep-4ecut}
	Let $ G = (V, E) $ be a counterexample to Theorem \ref{k33-thm}. 
	If $X \subset V$ satisfies $|X|, |V \setminus X| \ge 3$ and $d(X)=4$, then either $G[X]$ or $G[ V \setminus X]$ is a chain of sausages.
\end{lemma}

\begin{proof} First suppose that neither of the rooted graphs $G.X$, $G.(V \setminus X)$ has a rooted immersion of $W_4$.  In this case Corollary \ref{w4cor} implies that both of these graphs have a $(1,3)$-segmentation of width 4 where the first set of the segmentation is the singleton consisting of the root vertex.  It follows that $G$ has type 1, a contradiction.

Therefore, we may assume without loss of generality that the rooted graph $G.X$ satisfies $G.X \succ_r W_4$ and apply Lemma \ref{w4-eye-k33}.  If $G[X]$ is a chain of sausages we have nothing left to prove.  Otherwise $G. (V \setminus X)$ has a rooted immersion of $W_4$ and every vertex in $X$ has degree three.  Now we may apply Lemma \ref{w4-eye-k33} again to deduce that either $G[V \setminus X]$ is a chain of sausages or that every vertex in $V \setminus X$ has degree 3.  In the former case our proof is finished.  In the latter case our original graph $G$ is cubic, but then Kuratowski's Theorem implies that either $G$ is planar and type 0 or has a $K_{3,3}$ immersion, so Theorem \ref{k33-thm} holds for $G$ (a contradiction).
\end{proof}

We are now ready to establish the main result from this subsection.

\begin{lemma}
	\label{no-deep-4ecut}
	If $ G = (V, E) $ is a counterexample to Theorem \ref{k33-thm} with $|V|$ minimum, then every $X \subset V$ with $|X|, |V \setminus X| \ge 3$ satisfies $d(X) \ge 5$.
\end{lemma}
\begin{proof}
Suppose (for a contradiction) that the lemma does not hold, and let $G'$ be the graph obtained by sausage reducing $G$.  First we consider the possibility that $|V(G')| \le 5$, and we will show this forces $G$ to be type 1---a contradiction.  If $G'$ is a doubled cycle, then $G$ has type 1.  Next suppose that $G'$ has a unique sausage of order 2, say $G[Y]$.  In this case $X = V(G') \setminus Y$ satisfies $|X| \le 3$ and $d_G(X) = 4$ and there is a $(3,1)$-segmentation of $G$ relative to $(X, \{y\})$ for some vertex $y$ so we again find that $G$ is type 1. The only remaining possibility is that $G'$ has two sausages $G[X], G[Y]$ of order 2 where $X \cap Y = \emptyset$.  In this case $G'$ has a $(2,2)$-segmentation of width 4 relative to $(X,Y)$ and it follows easily that $G$ is type 1.  

It follows from Lemma \ref{deep-4ecut} and our assumptions that $6 \le |V(G')| < |V(G)|$ so Theorem \ref{k33-thm} may be applied to $G'$. If $G' \succ K_{3, 3}$ then 
$G \succ G'$ implies $G \succ K_{3,3}$, which is a contradiction.  Therefore $G'$ has one of types 0, 1, 2, 3 or 4.  Since sausage reduction has been applied to $G$ nontrivially to obtain $G'$, the graph $G'$ has a chain of sausages of order 2.  It follows that $G'$ is not type 0 or 4.  Lemma \ref{deep-4ecut} implies that $G'$ does not have a set $X \subset V(G')$ with $d_{G'}(X)=4$ and $|X|, |V \setminus X| \ge 3$ and it follows that $G'$ is not type 1.  If $G'$ is type 3, then it is isomorphic to one of the 20 graphs in Figure \ref{k33-exceptions-1}, but then $G$ must also be type 3, a contradiction.  

In the remaining case $G'$ is type 2.  So to complete the proof, it will suffice to show that applying the opposite of a sausage reduction to an arbitrary graph $H$ of type 2 results in another graph of type 2.  Assume that $H$ is type 2 relative to $W,W'$ and let $H[Y]$ be a chain of sausages of order 2.  Note that $d_H(W) = d_H(W') = 5$ so we must have $|Y \cap W|, |Y \cap W'| \le 1$.  If $H$ is type 2A or 2B, then expanding $H[Y]$ to a longer chain of sausages (i.e. the reverse of a sausage shortening) results in another graph of the same type.  So we may assume $H$ is type 2C.  If $Y \cap W = \emptyset = Y \cap W'$, then expanding $H[Y]$ to a longer chain of sausages results in another graph of type 2C, otherwise this expansion will result in a graph of type 2A.  
\end{proof}

\subsection{Computation for small graphs}

In this subsection, we describe a computational verification of Theorem \ref{k33-thm} for graphs on at most nine vertices.
We will call on the following simple observation to reduce our task to a finite number of graphs. For $u, v$ distinct vertices of a graph $G$, we denote by $e(u, v)$ the number of edges between $u, v$.

\begin{observation}
	\label{e-mult-cap}
	Let $G, H$ be graphs, and let $u, v \in V(G)$ satisfy $e(u, v) > |E(H)|$.  Then $G\succ H$ if and only if the graph $G'$ obtained from $G$ by deleting one copy of $uv$ edge satisfies $G' \succ H$.
\end{observation}

\begin{lemma}
	\label{k33-code}
	Theorem \ref{k33-thm} holds for every graph $G$ with $|V(G)| \le 9$.
\end{lemma}

\begin{proof}
Thanks to Lemma \ref{no-deep-4ecut} and Observation \ref{e-mult-cap} it suffices to establish the Theorem for all graphs $G$ satisfying the following properties:
\begin{itemize}
\item $6 \le |V(G)| \le 9$,
\item $G$ is 3-edge-connected and internally 4-edge-connected,
\item Every $X \subset V(G)$ with $|X|, |V(G) \setminus X| \ge 3$ satisfies $d(X) \ge 5$,  
\item Every parallel class has size at most 9.
\end{itemize}

This calculation was done in Sagemath and the code may be found on the arXiv.  Here is a high-level description of the algorithm, which is run for $6 \le n \le 9$.  

	\begin{description}
		\item[Step 1.] We take the list of all connected simple graphs on $n$ vertices, and filter out the ones which immerse $K_{3, 3}$.
		\item [Step 2.]
		For any graph $G $ surviving from Step 1, \verb|repair(G)| generates a list consisting of all edge-minimal multigraphs $G'$ such that:
		\begin{itemize}
			\item
			the underlying simple graph of $G'$ is $G$,
			\item 
			$G'$ is 3-edge-connected and internally 4-edge-connected,
			\item
			for any set $X \subset V(G')$ where $3 \le |X| \le \left \lfloor \frac{n} {2}\right \rfloor$ we have $d_{G'} (X) \ge 5$,
			\item
			$G'$ does not immerse $K_{3,3}$. 
		\end{itemize}
	\item[Step 3.] Suppose the simple connected graph $G$ is such that \verb|repair(G)| is nonempty. Let $\mathcal{G}_1=$ \verb|repair(G)|. Then, using $\mathcal{G}_1$, we generate $\mathcal{G}_2 =$ \verb|obstruction(G)| which is the list consisting of all multigraphs whose underlying simple graph is $G$, meet the edge-connectivity conditions that the graphs in $\mathcal{G}_1$ satisfy, have edge-multiplicity at most nine, and do not immerse $K_{3,3}$.
	\item[Step 4.] Every graph in $\mathcal{G}_2$ is tested if it has one of the types 0, 2, or is isomorphic to one of the graphs in Fig. \ref{k33-exceptions-1} or \ref{k33-exceptions-2}.	
	\end{description}
	The calculation is done rather fast. It took a desktop computer 25 minutes to do the calculation for every $n \in \{6, 7, 8\}$. However, the time spent on $n =9$ was considerably more. It took the computer one hour to carry out step 1, i.e. to check the nearly 262,000 connected simple graphs on nine vertices for a $K_{3, 3}$ immersion, thereby giving a list \verb|N9| of almost 34,100 simple connected graphs on nine vertices without a $K_{3, 3}$ immersion. Then a total of four hours was spent on carrying out steps 2, 3 for every graph in \verb|N9|. Since no obstruction is found for $n =9$, step 4 is not performed for this case.
\end{proof}

\subsection{Five edge cuts}

In this subsection we study 5-edge-cuts in a minimal counterexample to Theorem \ref{k33-thm}. The main result is Lemma \ref{no-deep-5ecut}, which tells us that one side of every $5$-edge-cut in a minimal counterexample  has at most three vertices.  The lemma is a powerful tool in carrying out the inductive step in the proof of  the main theorem.  As a first step toward this, we record two local properties of a minimal counterexample which we frequently apply.

\begin{lemma}
	\label{k33-local}
	If $ G = (V, E) $ is a counterexample to Theorem \ref{k33-thm} with $|V|$ minimum, then:
	\begin{enumerate}[label=(\arabic*)]
		\item
		\label{no-greedy-nbr}
		There does not exist $u\in V$ which has a neighbour $v$ such that $e (u, v)\ge \frac{1}{2} d(u)$.
		\item 
		\label{4ecut}
		If $X \subset V$ satisfies $2 \le |X| \le \frac{1}{2}|V|$ and $d(X) = 4$, then $|X|=2$ and both vertices in $X$ have degree three.
	\end{enumerate}
\end{lemma}

\begin{proof}
For part (1), suppose for a contradiction that such $u, v$ exist and let $G' = G.\{u,v\}$.  Since  $|V(G')| < |V(G)|$, Theorem \ref{k33-thm} holds for $G'$.  The graph $G'$ cannot have an immersion of $K_{3,3}$ as otherwise $K_{3,3} \prec G' \prec G$.  Therefore $G'$ must have type 0, 1, 2, 3, or 4.  Type 0 is impossible since $G'$ cannot be cubic (the vertex formed by identifying $u$ and $v$ has degree at least 4).  It follows from Lemma 
\ref{no-deep-4ecut} that $G'$ does not have a set $X \subseteq V(G')$ with $|X|, |V(G') \setminus X| \ge 3$ and $d(X) = 4$ and thus $G'$ cannot be type 1.  It follows from this same property that $G'$ must be sausage reduced, but then the lower bound $9 \le |V(G')|$ (implied by Lemma \ref{k33-code}) prevents $G'$ from having types 2, 3, or 4.  This contradiction completes the proof of \ref{no-greedy-nbr}.  

Part \ref{4ecut} is an immediate consequence of Lemma \ref{no-deep-4ecut} and part \ref{no-greedy-nbr}.
\end{proof}

Next we introduce a bit of convenient terminology.

\begin{definition}
	\normalfont
	Let $G$ be a graph, let $X \subseteq V(G)$, let $e \in \delta(X)$, and let $x$ be the endpoint of $e$ in $X$.  We say that $X$ is \emph{almost cubic relative to} $e$ if $d(x) \in \{3,4\}$ and $d(u) = 3$ for every $u \in X \setminus \{x\}$.
\end{definition}

We now begin our investigation of 5-edge-cuts in a minimum counterexample by establishing a technical condition.

\begin{lemma} 
	\label{5ecut-e}
Let $G$ be a counterexample to Theorem \ref{k33-thm} with $|V|$ minimum.  Let $\delta(X)$ be a $5$-edge-cut such that $|X|, |V\setminus X| \ge 4$. If $e\in \delta(X)$ satisfies $(G\setminus e).X\succ_r W_4$, then:
	\begin{itemize}
		\item 
		$(G\setminus e).(V\setminus X)\succ_r W_4$,
		\item
		Both $X$ and $V \setminus X$ are almost cubic relative to $e$.
	\end{itemize}
\end{lemma}

\begin{proof}
	Let $H$ be the graph obtained from $G$ by deleting $e$ and then suppressing any vertices of degree two.  It follows from the internal 4-edge-connectivity of $G$ that $H$ is 3-edge-connected, and it follows from Lemma \ref{k33-local}\ref{4ecut} that $H$ is internally 4-edge-connected.  Since $G$ immerses $H$ we must have $H \nsucc K_{3,3}$.  Let $X' \subset V(H)$ be the subset of $H$ corresponding to $X$ (i.e. either $X' = X$ or $X'$ is obtained from $X$ by removing a vertex that was suppressed in forming $H$).  It now follows from Lemma \ref{w4-eye-k33} that one of the following holds:
	\begin{itemize}
		\item
		$H[X']$ is a chain of sausages (in $H$)
		\item 
		for every vertex $v \in X'$ we have $d_H(v) = 3$, and $H. (V(H) \setminus X') \succ_r W_4$ 
	\end{itemize}
The former case is not possible since this would cause $G$ to have a degree four vertex incident with parallel edges contradicting Lemma \ref{k33-local}\ref{no-greedy-nbr}.  So we must have the latter case.  This implies that our original graph satisfies $(G \setminus e).(V \setminus X) \succ_r W_4$ and that $X$ is almost cubic relative to $e$.  Now we may apply the same argument with $V \setminus X$ in place of $X$ to deduce that $V \setminus X$ is also almost cubic relative to $e$, and this completes the proof.
\end{proof}

Next we prove our main result from this subsection.  

\begin{lemma}
	\label{no-deep-5ecut}
	If $G = (V,E)$ is a counterexample to Theorem \ref{k33-thm} with $|V|$ minimum, there does not exist
	$X \subset V(G)$ so that $|X|, |V\setminus X| \ge 4$ and $d(X) \le 5$.
\end{lemma}
\begin{proof}
Suppose (for a contradiction) that such a set $X$ exists, and note that by Lemma \ref{no-deep-4ecut} we must have $d(X) = 5$.  It follows from Lemma \ref{k33-code} that $|V| \ge 10$ and therefore we may assume (by possibly interchanging $X$ and $V \setminus X$) that $G.X$ has at least 6 vertices.  It follows from Lemma \ref{k33-local}\ref{no-greedy-nbr} that $G.X$ does not have  a vertex of degree four which is not incident with the vertex $X$ and has less than four neighbours.  Thus Corollary \ref{w4cor} implies that $G.X$ has a rooted immersion of $W_4$.  Choose $e \in \delta(X)$ so that $(G-e).X$ has a rooted immersion of $W_4$.  Now Lemma \ref{5ecut-e} implies that $(G-e).(V \setminus X)$ also has a rooted immersion of $W_4$ and that both $X$ and $V \setminus X$ are almost cubic relative to $e$.  Let $S$ ($T$) denote the set of all vertices in $X$ ($V \setminus X$) incident with an edge in $\delta(X)$ and observe that we must have $|S|, |T| \ge 4$.  

The graph $G$ has at most two vertices with degree greater than three, so it follows from Kuratowski's Theorem that $G$ is planar (otherwise $G$ would have a $K_{3,3}$ subdivision and hence a $K_{3,3}$ immersion).  Since $G[X]$ may be obtained from $G.(V \setminus X)$ by deleting the root, this graph has an embedding in the plane for which all vertices in $S$ are incident with the unbounded face.  It follows from the internal 4-edge-connectivity of $G$ and Lemma \ref{no-deep-4ecut} that $G[X]$ does not have a cut-edge.  This together with the fact that $G[X]$ has maximum degree at most 3 implies that it is 2-connected.  Therefore, the unbounded face in our planar embedding is bounded by a cycle $C$ with $S \subseteq V(C)$.  A similar argument shows that $G[ V \setminus X]$ has a cycle $D$ with $T \subseteq V(D)$.  

First suppose that $|T| = 5$.  In this case the existence of the cycle $D$ implies that $(G-f).X$ has a rooted immersion of $W_4$ for every $f \in \delta(X)$.  Now Lemma \ref{5ecut-e} implies that both $X$ and $V \setminus X$ are almost cubic relative to $f$ for every $f \in \delta(X)$.  This forces $G$ to be a cubic graph, and now Kuratowski's Theorem implies that $G$ is type 0 giving us a contradiction.  A similar argument handles the case when $|S| = 5$.  So we may assume $|S| = |T| = 4$.  In this case there is a unique vertex $s \in S$ ($t \in T$) incident with two edges of $\delta(X)$.  Moreover, for every $f \in \delta(X)$ the graph $(G-f).(V \setminus X)$ ($(G-f).X$) has a rooted immersion of $W_4$ if and only if $f$ is incident with $s$ ($t$).  It now follows from Lemma \ref{5ecut-e} that $s$ and $t$ are joined by two parallel edges.  However, this contradicts Lemma \ref{k33-local}\ref{no-greedy-nbr}, thus completing our proof.
\end{proof}

\subsection{Finishing the proof}

\label{k33-pf-end}

In this subsection, we combine our lemmas to complete a proof of the main theorem.

\begin{proof}[Proof of Theorem \ref{k33-thm}]
Suppose (for a contradiction) that Theorem \ref{k33-thm} is false and choose a counterexample $G = (V,E)$ so that
\begin{enumerate}[label=(\roman*)]
\item $|V|$ is minimum.
\item $|E|$ is minimum subject to (i).
\end{enumerate}
Note that Lemmas \ref{no-deep-4ecut} and \ref{no-deep-5ecut} imply that every 4-edge-cut of $G$ has one side of size at most two and every 5-edge-cut of $G$ has one side of size at most three.  Lemma \ref{k33-code} implies that $|V| \ge 10$ and Lemma \ref{k33-local} implies that no vertex has at least half of its incident edges in a common parallel class.  

First suppose that there exist vertices $u,v$ with $e(u,v) \ge 2$ and let $G'$ be the graph obtained from $G$ by deleting one edge between $u$ and $v$.  It follows from Lemma \ref{k33-local} that $G'$ is 3-edge-connected and internally 4-edge-connected.  So, by the minimality of our counterexample, the Theorem applies to $G'$ and it must have type 0, 1, 2, 3, or 4.  Since $G'$ is sausage reduced and has at least 10 vertices, types 2, 3, and 4 are impossible.  If $G'$ is type 0, it is cubic and simple, and thus $G$ contradicts Lemma \ref{k33-local}\ref{4ecut}. If $G'$ is type 1, we get a contradiction with either Lemma \ref{k33-local}\ref{4ecut} or \ref{no-deep-5ecut}. 
Therefore our graph $G$ must be simple.  

If $G$ is cubic, then Kuratowski's Theorem implies that $G$ is either type 0 or has an immersion of $K_{3,3}$.  So there must exist a vertex $v \in V(G)$ with $d(v) \ge 4$.  Choose an edge $e \in E$ not incident with $v$ and let $G'$
be the graph obtained from $G$ by deleting $e$, and suppressing any degree two vertices.  So $|V(G')|\ge |V(G)| -2 \ge 8$, and $G'$ has minimum degree at least 3.  It follows from Lemma \ref{k33-local}\ref{4ecut} that $G$ is 3-edge-connected and internally 4-edge-connected.  So by minimality of $G$, Theorem \ref{k33-thm} holds for $G'$, and it must be type 0, 1, 2, 3, or 4.  Since $G'$ is not cubic, it is not type 0.  As in the last paragraph, $G'$ cannot have type 1 as then $G$ would contradict either Lemma \ref{k33-local}\ref{4ecut} or \ref{no-deep-5ecut}.  
Note that it follows from Lemma \ref{no-deep-4ecut} and $G$ being simple that the graph $G'$ is sausage reduced. This together with the bound $|V(G')|\ge 8$ and $G$ being simple imply that $G'$ cannot be type 2B, 2C, 3 or 4. 
In the only remaining case, $G'$ has type 2A relative to some $W,W'$ with $|W| = |W'| = 2$.  However this is incompatible with the assumption that $G$ is simple and Lemma \ref{no-deep-5ecut}.  This final contradiction completes the proof.
\end{proof}

\subsection{Corollaries of the main theorem}  
\label{sub-k33-cor}

In this subsection, we establish a couple consequences of our main theorem.  We prove a bound on the branch-width and path-width of graphs of type 1 which combines with our main theorem to give a best possible bound on the branch-width of suitably connected graphs with no $K_{3,3}$ immersion that are not cubic and planar.  Our main theorem also has an immediate corollary for graphs with no $K_{3,3}$ and $K_5$ immersion and together with our branch-width lemma this permits us a best-possible improvement in the branch-width bound from Theorem \ref{intro-k33-greek}.

%
%
%

\subsubsection*{Branch-width of graphs without $K_{3,3}$ immersion}

If $G=(V,E)$ is a graph and $A,B \subseteq E$ satisfy $A \cup B = E$ and $A \cap B = \emptyset$ then we call $(A,B)$ a \emph{separation}.  The \emph{order} of this separation is the number of vertices that are incident with both an edge in $A$ and an edge in $B$, and we denote the order by $o(A,B)$.  A \emph{branch-decomposition} of $G$ consists of a cubic tree (a tree where every vertex either has degree three or one) $T$ together with an injective mapping $f$ from $E$ to leaves of $T$.  For every edge $e \in E(T)$ the graph $T-e$ has two components, say $T_1, T_2$.  Let $A$ $(B)$ be the set of edges in $E$ which are mapped by $f$ to a leaf in $T_1$ $(T_2)$, and define $w(e) = o(A,B)$.   The \emph{branch-width} of this branch decomposition is the maximum of $w(e)$ taken over all edges of $T$.  The \emph{branch-width} of $G$, denoted $bw(G)$, is the minimum width of a branch decomposition of $G$.  Our main result from this section will give a sharp bound on the branch width for graphs of type 1.  First let us record a key observation about these graphs.

\begin{observation}
\label{seg-nbr}
Let $G$ be a $3$-edge-connected graph with a $(3, 3)$-segmentation of width four $U_0 \subset U_1 \subset \ldots \subset U_t$. Let $U_i \setminus U_{i-1} = \{x_i\}$, and let $Z_i$ be the set of vertices in $U_i$ that are incident with an edge in $\delta (U_i)$, for $i= 1, \ldots , t$. Then for every $1 \le i \le t$ we have:
\begin{enumerate}[label=(\arabic*)]
\item 
\label{xi-d}
$e(x_i, U_{i-1}) = e(x_i, V \setminus U_i) \ge 2$
\item 
\label{Zile3}
$x_i \in Z_i$ and $|Z_i| \le 3$.
\end{enumerate}	
\end{observation}
%

Next we prove a lemma that exhibits the key structure of interest.  
%
%
%
%
%
%
%
%

\begin{lemma}
\label{type1-par-cw}
Let $G$ be a graph with a $(3,3)$-segmentation of width four. Then there exists a partition of $E=E(G)$ into $\{E_0, E_1, \ldots , E_k\}$ such that
\begin{itemize}
\item 
Every $E_i$ with $1 \le i \le k-1$ is a parallel class
\item
$o \Big( \bigcup_{i=0}^j E_i , \bigcup_{i=j+1}^k E_i \Big) \le 3$ for $0 \le i \le k-1$.
\end{itemize}
\end{lemma}

\begin{proof}
Let $U_0 \subset U_1 \subset U_2 \subset \ldots \subset U_t$ be a $(3,3)$-segmentation of width four of $G$. Let $S$ be a union of parallel classes of $G$. We say $S$ has a {\it good ordering} if there exists a partition of $S$ into parallel classes $\{E_1, \ldots , E_k\}$ such that 
\begin{itemize}
\item $o \Big( \bigcup_{i=1}^j E_i , E \setminus \bigcup_{i=1}^j E_i \Big) \le 3$ holds for every $1 \le j \le k$
\end{itemize}
Note that in order to prove the Lemma, it suffices to show that $E(G)$ has a good ordering.
Clearly, there is a good ordering of $E[G(U_0)]$. If $t=0$, there is nothing left to prove, so  suppose $t\ge 1$. Let $U_i \setminus U_{i-1} = \{x_i\}$, and let 
$Z_i$ be the set of vertices in $U_i$ that are adjacent with an edge in $\delta (U_i)$, for $i= 1, \ldots , t$.  Suppose there exists a good ordering $\{E_0, \ldots, E_l\}$ of $E(G[U_{i-1}])$, and we extend this to a good ordering for $E(G[U_i])$ by appending the parallel classes in $E(x_i, U_{i-1})$ to it in a certain order. Note that by Observation \ref{seg-nbr}\ref{Zile3} we have $|Z_{i}|\le 3$. If $|Z_{i-1}|\le 2$,  adding the parallel classes contained in $E(x_i, U_{i-1})$ in any order to $\{E_1, \ldots, E_l\}$ results in a good ordering. If $|Z_{i-1}| =3$, it follows from Observation \ref{seg-nbr}\ref{xi-d} that there exists $v \in Z_{i-1}\setminus Z_i$. Now, we let $E_{l+1}$ be the parallel class of $E(v, x_i)$, and then we add to $\{E_0, \ldots, E_{l+1}\}$ the other parallel classes between $x_i, U_{i-1}$ (if any) in an arbitrary order. This gives a good ordering of $E(G[U_i])$.
\end{proof}

The above lemma  implies that whenever $G$ has a $(3,3)$-segmentation of type 1, the underlying simple graph of $G$ has a branch decomposition of width 3 where the tree is a caterpillar.  This lemma can also be used to show that $G$ has path-width at most 3, but we will not delve further into this matter.  Our main interest is the following theorem which is an immediate consequence.

\begin{theorem}
If $G$ has a $(3,3)$-segmentation of width four, then $bw(G) \le 3$.  
\end{theorem}

It is easy to see that if $H$ is the underlying simple graph of $G$ and $|E(H)| \ge 2$, then $bw(G) = bw (H)$. Also, if $H$ is a subdivision of $G$ and $bw (G) \ge 2$, then $bw (G) = bw (H)$. Therefore, if $G$ is a graph which is obtained from another graph $H$ by sausage reducing it, then $bw (G) = bw (H)$. So, in order to find the branch-width of graphs of types 2, 3, or 4, it suffices to consider only the ones which are sausage reduced, and have distinct underlying simple graphs.  It is then easy to check that all such graphs have branch-width at most three, except for the Octahedron (which is known to have branch-width four).  The following corollary is an immediate consequence of this discussion.

\begin{corollary}
	Let $G$ be a graph which does not immerse $K_{3,3}$. Then $G$ can be constructed from $i$-edge-sums, for $i = 1, 2, 3$ from planar cubic graphs, Octahedron, and graphs with branch-width at most $3$.
\end{corollary}


\subsubsection*{Forbidding $K_{3,3}$ and $K_5$}
Observe that among graphs of types 0-4, there is only one which immerses $K_5$---Octahedron. Therefore, as a corollary of Theorem \ref{k33-thm} we obtain the following characterizations of the graphs which exclude both $K_{3,3}$ and $K_5$ as immersion, the second of which strengthens the earlier theorem of Giannopoulou et al. (Theorem \ref{intro-k33-greek}).
\begin{corollary}
	\label{k33-k5}
	Let $G$ be a $3$-edge-connected, internally $4$-edge-connected graph with $|V(G)|\ge 6$ that does not immerse $K_{3,3}$ or $K_5$.  Then
	\begin{itemize}
		\item 
		$G$ has one of the types $0, 1, 2, 3$, or $4$ except for the Octahedron.
		\item 
		$G$ is either planar and cubic, or has branch-width at most three.
	\end{itemize}
\end{corollary}
\section*{Acknowledgment}
The authors would like to thank Jessica McDonald who worked with them at the start of this project and contributed enormously to their understanding of the problem. The second author gratefully acknowledges Stefan Hannie's valuable help with the code used in this project.
\bibliographystyle{plain}
\bibliography{references}
\end{document}